\documentclass[a4paper,11pt]{article}

\def\RR{\mathbb{R}}
\def\ZZ{\mathbb{Z}}

\usepackage{amsmath}
\usepackage{amsfonts}

\newtheorem{lemma}{Lemma}
\newtheorem{proposition}{Proposition}
\newtheorem{theorem}{Theorem}

\newenvironment{proof}
 {\begin{trivlist} \item[] {\bf Proof.\ }}{\hfill$\Box$ \end{trivlist}}
\newenvironment{proofof}[1]%
 {\begin{trivlist} \item[] {\bf #1.\ }}{\hfill$\Box$ \end{trivlist}}

\begin{document}

\title{Compact Formulations of the Steiner Traveling Salesman Problem and
Related Problems}

\author{Adam N.\ Letchford\thanks{Corresponding author. Department of Management
Science, Lancaster University, Lancaster LA1 4YX, United Kingdom. E-mail:
{\tt A.N.Letchford@lancaster.ac.uk}} \and Saeideh D.\ Nasiri\thanks{Lancaster
University. E-mail: {\tt s.d.nasiri@lancaster.ac.uk}} \and Dirk Oliver
Theis\thanks{Faculty of Mathematics, Otto von Guericke University of
  Magdeburg, Germany.  E-mail: {\tt theis@ovgu.de}}}

\date{Sat Mar 17 12:04:35 CET 2012}

\maketitle

\begin{abstract}
The Steiner Traveling Salesman Problem (STSP) is a variant of the Traveling
Salesman Problem (TSP) that is particularly suitable when dealing with sparse
networks, such as road networks. The standard integer programming formulation
of the STSP has an exponential number of constraints, just like the standard
formulation of the TSP. On the other hand, there exist several known
{\em compact}\/ formulations of the TSP, i.e., formulations with a polynomial
number of both variables and constraints. In this paper, we show that some of
these compact formulations can be adapted to the STSP. We also briefly discuss
the adaptation of our formulations to some closely-related problems.
\\*[2mm]
\end{abstract}

\section{Introduction}

The {\em Traveling Salesman Problem}\/ (TSP), in its undirected version, can be
defined as follows. We are given a complete undirected graph $G=(V,E)$ and a
positive integer cost $c_e$ for each edge $e \in E$. The task is to find a
Hamiltonian circuit, or {\em tour}, of minimum total cost. The best algorithms
for solving the TSP to proven optimality, such as the ones described in
\cite{ABCC06,N02,PR91}, are based on a formulation of the TSP as a 0-1 linear
program due to Dantzig {\em et al.}\ \cite{DFJ54}, which we present in
Subsection \ref{sub:lit1} of this paper.

The Dantzig {\em et al.}\ formulation has only one variable per edge, but has
an exponentially-large number of constraints, which makes cutting-plane methods
necessary (see again \cite{ABCC06,N02,PR91}). If one wishes to avoid this
complication, one can instead use a so-called {\em compact}\/ formulation of
the TSP, i.e., a formulation with a polynomial number of both variables and
constraints. A variety of compact formulations are available (see the surveys
\cite{GV95,OAL09,OW07,PS91} and also Subsection \ref{sub:lit2} of this paper).

When dealing with routing problems on real-life road networks, however, one
often encounters the following variant of the TSP. The graph $G$ is not
complete, not every node must be visited by the salesman, nodes may be visited
more than once if desired, and edges may be traversed more than once if
desired. This variant of the TSP was proposed, apparently independently, by
three sets of authors \cite{CFN85,F85,O74}. (The special case in which all
nodes must be visited was considered earlier in \cite{HN62,MLN81}.) We will
follow Cornu\'ejols {\em et al.}\ \cite{CFN85} in calling this variant the
{\em Steiner}\/ TSP, or STSP for short.

As noted in \cite{CFN85,F85}, it is possible to convert any instance of the
STSP into an instance of the standard TSP, by computing shortest paths between
every pair of required nodes. So, in principle, one could use any of the
above-mentioned TSP formulations to solve the STSP. If, however, the original
STSP instance is defined on a sparse graph, the conversion to a standard TSP
instance increases the number of variables substantially, which may be
undesirable. For this reason, we have decided in this paper to present and
analyse some compact formulations for the STSP.

The paper is structured as follows. We review the relevant literature on TSP
and STSP formulations in Section \ref{se:lit}. In Section \ref{se:flow}, we
show how to adapt so-called {\em commodity-flow}\/ formulations of the TSP to
the Steiner case, and make some remarks about the relative strength of the
resulting formulations. In Section \ref{se:time}, we adapt the so-called
{\em time-staged}\/ formulation of the TSP to the Steiner case, and present
a key theorem, which enables one to reduce the number of variables
substantially. Then, in Section \ref{se:other}, we briefly discuss the
possibility of adapting our compact formulations to some other vehicle
routing problems, when sparse graphs are involved rather than complete graphs.
Finally, some concluding remarks appear in Section
\ref{se:conclusion}.

\section{Literature Review} \label{se:lit}

We now review the relevant literature. We cover the classical formulation
of the standard TSP in Subsection \ref{sub:lit1}, compact formulations of the
standard TSP in Subsection \ref{sub:lit2}, and the classical formulation of the
STSP in Subsection \ref{sub:lit3}.

\subsection{The classical formulation of the standard TSP} \label{sub:lit1}

The classical and most commonly-used formulation of the standard TSP is the
following one, due to Dantzig, Fulkerson and Johnson \cite{DFJ54}:
\begin{eqnarray}
\nonumber
\min           & \sum_{e \in E} c_e x_e             &\\
\label{eq:deg}
\mbox{s.t.}    & \sum_{e \in \delta(\{i\})} x_e = 2 & (\forall i \in V) \\
\label{eq:SEC} & \sum_{e \in \delta(S)} x_e \ge 2   & (\forall S \subseteq V: 2 \le |S| \le |V|/2)\\
\nonumber      & x_e \in\{0,1\}                     & (\forall e \in E).
\end{eqnarray}
Here, $x_e$ is a binary variable, taking the value $1$ if and only if the
edge $e$ belongs to the tour, and, for any $S \subset V$, $\delta(S)$ denotes
the set of edges having exactly one end-node inside $S$. The constraints
(\ref{eq:deg}), called {\em degree}\/ constraints, enforce that the tour uses
exactly two of the edges incident on each node. The constraints (\ref{eq:SEC}),
called {\em subtour elimination constraints}, ensure that the tour is connected.

We will call this formulation the {\em DFJ}\/ formulation. A key feature of this
formulation is that the subtour elimination constraints (\ref{eq:SEC}) are exponential
in number.

\subsection{Compact formulations of the standard TSP} \label{sub:lit2}

As mentioned above, a wide variety of compact formulations exist for the
standard TSP, and there are several surveys available (e.g.,
\cite{GV95,OAL09,OW07,PS91}). For the sake of brevity, we mention here only
four of them. All of them start by setting $V=\{1,2,\ldots n\}$ and viewing
node $1$ as a `depot', which the salesman must leave at the start of the tour
and return to at the end of the tour. Moreover, all of them can be used for
the asymmetric TSP as well as for the standard (symmetric) TSP.

We begin with the formulation of Miller, Tucker \& Zemlin \cite{MTZ60}, which
we call the {\em MTZ}\/ formulation. For all node pairs $(i,j)$, let
${\tilde x}_{ij}$ be a binary variable, taking the value $1$ if and only if
the salesman travels from node $i$ to node $j$. Also, for $i = 2, \ldots, n$,
let $u_i$ be a continuous variable representing the position of node $i$ in
the tour. (The depot can be thought of as being at positions $0$ and $n$.)
The MTZ formulation is then:
\begin{eqnarray}
\label{eq:dirobj}
\min               & \sum_{i,j=1}^n c_{ij} {\tilde x}_{ij}     & \\
\label{eq:assign1}
\mbox{s.t.}        & \sum_{j=1}^n {\tilde x}_{ji}=1            & (1 \le i \le n) \\
\label{eq:assign2} & \sum_{j=1}^n {\tilde x}_{ij}=1            & (1 \le i \le n) \\
\label{eq:dirbin}  & {\tilde x}_{ij} \in \{0,1\}               & (1 \le i,j \le n; i \ne j)  \\
\label{eq:MTZ1}    & u_i - u_j + (n-1) {\tilde x}_{ij} \le n-2 & (2 \le i,j \le n; i \ne j)\\
\label{eq:MTZ2}    & 1 \le u_i \le n-1                         & (2 \le i \le n).
\end{eqnarray}
The constraints (\ref{eq:assign1}) and (\ref{eq:assign2}) ensure that the salesman
arrives at and departs from each node exactly once. The constraints (\ref{eq:MTZ1})
ensure that, if the salesman travels from $i$ to $j$, then the position of node $j$
is one more than that of node $i$. Together with the bounds (\ref{eq:MTZ2}), this
ensures that each non-depot node is in a unique position.
 
The MTZ formulation is compact, having only ${\cal O}(n^2)$ variables and
${\cal O}(n^2)$ constraints. Unfortunately, Padberg \& Sung \cite{PS91} show
that its LP relaxation yields an extremely weak lower bound, much weaker than that
of the DFJ formulation.

The next compact formulation, historically, was the `time-staged' (TS)
formulation proposed by both Vajda \cite{V61} and Houck {\em et al.} \cite{HPQV80} 
independently. For all $1 \le i,j,k \le n$ with $i \ne j$, let $r_{ij}^k$ be 
a binary variable taking the value $1$ if and only if the edge $\{i,j\}$ is 
the $k$th edge to be traversed in the tour, and is traversed in the direction 
going from $i$ to $j$. We then have:
\begin{eqnarray}
\nonumber
\min              & \sum_{i=2}^n c_{1i} r_{1i}^1 + \sum_{k=2}^{n-1} \sum_{i,j=2}^{n}
                    c_{ij} r_{ij}^k + \sum_{i=2}^n c_{i1} r_{i1}^{n} & \\
\mbox{s.t.}       & \hspace{-1cm} \sum_{j=2}^n r_{1j}^1 = 1 & \label{eq:out1}\\
\label{eq:in1}    & \hspace{-1cm} \sum_{j=2}^n r_{j1}^n = 1 & \\
\label{eq:ini}    & \hspace{-1cm} \sum_{k=1}^{n-1} \sum_{j \ne i} r_{ji}^k = 1
                  & \hspace{-2.5cm} (2 \le i \le n) \\
\label{eq:inouti} & \hspace{-1cm}\sum_{j \ne i} r_{ji}^k = \sum_{j \ne i} r_{ij}^{k+1}
                  & \hspace{-2.5cm} (2 \le i \le n; 1 \le k \le n-1)\\
\nonumber         & \hspace{-1cm}r_{ij}^k \in \{0,1\} & \hspace{-2.5cm} (1 \le i,j,k \le n; i \ne j).
\end{eqnarray}
The constraints (\ref{eq:out1}) and (\ref{eq:in1}) state that the salesman must
leave the depot at the start of the tour and return to it at the end. The constraints
(\ref{eq:ini}) ensure that the salesman arrives at each non-depot node exactly once,
and the constraints (\ref{eq:inouti}) ensure that the salesman departs from each node
that he visits.

The TS formulation has ${\cal O}(n^3)$ variables and ${\cal O}(n^2)$ constraints.
It follows from results in \cite{GV95,PS91} that the associated lower bound is
intermediate in strength between the MTZ and DFJ bounds.

Next, we mention the {\em single-commodity flow}\/ (SCF) formulation of Gavish
\& Graves \cite{GG78}. Imagine that the salesman carries $n-1$ units of a
commodity when he leaves node $1$, and delivers $1$ unit of this commodity to
each other node. Let the ${\tilde x}_{ij}$ variables be defined as above, and
define additional continuous variables $g_{ij}$, representing the amount of the
commodity (if any) passing directly from node $i$ to node $j$. The formulation
then consists of the objective function (\ref{eq:dirobj}), the constraints
(\ref{eq:assign1})--(\ref{eq:dirbin}), and the following constraints:
\begin{eqnarray}
\label{eq:gflow}   & \sum_{j=1}^n g_{ji} - \sum_{j=2}^n g_{ij} = 1 & (2 \le i \le n)\\
\label{eq:gbounds} & 0 \le g_{ij} \le (n-1) {\tilde x}_{ij} & (1 \le i,j \le n; j \ne i).
\end{eqnarray}
The constraints (\ref{eq:gflow}) ensure that one unit of the commodity is delivered
to each non-depot node. The bounds (\ref{eq:gbounds}) ensure that the commodity
can flow only along edges that are in the tour.

The SCF formulation has ${\cal O}(n^2)$ variables and ${\cal O}(n)$ constraints.
It is proved in \cite{PS91} that the associated lower bound is intermediate in
strength between the MTZ and DFJ bounds. Later on, in \cite{GV95}, it was shown
that it is in fact intermediate in strength between the MTZ and TS bounds.

Finally, we mention the {\em multi-commodity flow}\/ (MCF) formulation of Claus
\cite{C84}. Here, we imagine that the salesman carries $n-1$ commodities, one
unit of each for each customer. Let the ${\tilde x}_{ij}$ variables be defined
as above. Also define, for all $1 \le i, j \le n$ with $i \ne j$ and all
$2 \le k \le n$, the additional continuous variable $f_{ij}^k$, representing the
amount of the $k$th commodity (if any) passing directly from node $i$ to node $j$.
The formulation then consists of the objective function (\ref{eq:dirobj}), the
constraints (\ref{eq:assign1})--(\ref{eq:dirbin}), and the following constraints:
\begin{eqnarray}
\label{eq:MCF1}
 &  0 \le f_{ij}^k \le {\tilde x}_{ij} & (k = 2, \ldots, n; \{i,j\} \subset \{1, \ldots, n\})\\
\label{eq:MCF2}
  & \sum_{i=2}^n f_{1i}^k = 1 & (k = 2, \ldots, n)\\
\label{eq:MCF3}
  & \sum_{i=1}^n f_{ik}^k = 1 & (k = 2, \ldots, n)\\
\label{eq:MCF4}
  & \sum_{i=1}^n f_{ij}^k - \sum_{i=2}^n f_{ji}^k = 0 & (k = 2,\ldots,n; j \in \{2, \ldots, n\} \setminus \{k\}).
\end{eqnarray}
The constraints (\ref{eq:MCF1}) state that a commodity cannot flow along an
edge unless that edge belongs to the tour. The constraints (\ref{eq:MCF2}) and
(\ref{eq:MCF3}) impose that each commodity leaves the depot and arrives at its
destination. The constraints (\ref{eq:MCF4}) ensure that, when a commodity
arrives at a node that is not its final destination, then it also leaves that
node.

The MCF formulation has ${\cal O}(n^3)$ variables and ${\cal O}(n^3)$
constraints. It is proved in \cite{PS91} that the associated lower bound is
equal to the DFJ bound. Therefore, this is the strongest of the four compact
formulations mentioned.

\subsection{The classical formulation of the STSP} \label{sub:lit3}

In the STSP, $G=(V,E)$ is permitted to be a general graph, and a set
$V_R \subset V$ of {\em required nodes}\/ is specified. The formulation given
in \cite{F85} is as follows:
\begin{eqnarray}
\label{eq:stsp-obj}
\min              & \sum_{e \in E} c_e x_e & \\
\label{eq:con}
\mbox{s.t.}       & \sum_{e \in \delta(S)} x_e \ge 2 & 
    (S \subset V: S \cap V_R \ne \emptyset, V_R \setminus S \ne \emptyset)\\
\label{eq:parity} & \sum_{e \in \delta(i)} x_e \mbox{ even} & (i \in V)\\
\label{eq:genint} & x_e \in \ZZ_+ & (e \in E).
\end{eqnarray}
Note that the $x$ variables are now general-integer variables. Note also that the
parity conditions (\ref{eq:parity}) are non-linear. (They can be easily linearised,
using one additional variable for each node.) The crucial point, however, is that
there are an exponential number of the connectivity constraints (\ref{eq:con}).

\section{Flow-Based Formulations of the STSP} \label{se:flow}

In this section, we adapt the formulations SCF and MCF, mentioned in Subsection
\ref{sub:lit2}, to the Steiner case. We also give some results concerned with
the strength of the LP relaxations of our formulations.

\subsection{Some notation and a useful lemma} \label{sub:notation}

At this point, we present some additional notation. Let ${\tilde G} = (V,A)$
be a directed graph, where the set of directed arcs $A$ is obtained from the
edge set $E$ by replacing each edge $\{i,j\}$ with two directed arcs $(i,j)$ and
$(j,i)$. For each arc $a \in A$, the cost $c_a$ is viewed as being equal to the
cost of the corresponding edge. For any node set $S \subset V$, let $\delta^+(S)$
denote the set of arcs in $A$ whose tail is in $S$ and whose head is in
$V \setminus S$, and let $\delta^-(S)$ denote the set of arcs in $A$ for which
the reverse holds. For readability, we write $\delta^+(i)$ and $\delta^-(i)$ in
place of $\delta^+(\{i\})$ and $\delta^-(\{i\})$, respectively. Finally, let
$n_R = |V_R|$ denote the number of required nodes.

We will find the following lemma useful:
\begin{lemma} \label{le:once}
In an optimal solution to the STSP, no edge will be traversed more
than once in either direction.
\end{lemma}
This lemma is part of the folklore, but an explicit proof can be found in the
appendix of \cite{LO09}.

Using this fact, one can define a binary variable ${\tilde x}_a$ for each arc
$a \in A$, taking the value $1$ if and only if the salesman travels along $a$.

\subsection{An initial single-commodity flow formulation} \label{sub:newSCF1}

Without loss of generality, assume that node $1$ is required. By analogy with
the case of the standard TSP, we imagine that the salesman departs the depot with
$n_R-1$ units of the commodity, and delivers one unit of that commodity to each
required node. So, for each arc $a \in A$, let the new variable $g_a$ represent
the amount of the commodity passing through $a$. The single-commodity flow
formulation (SCF) may then be adapted to the sparse graph setting as follows:
\begin{eqnarray}
\label{eq:sparseSCFobj}
\min        & \hspace{-3.0cm} \sum_{a \in A} c_a {\tilde x}_a \\
\mbox{s.t.} & \label{eq:Xbounds}\hspace{-1.8cm} \sum_{a\in\delta^+(i)} {\tilde x}_a \ge 1 & (\forall i\in V_R)\\
\label{eq:xFlow}
            & \sum_{a \in \delta^+(i)} {\tilde x}_a = \sum_{a \in \delta^-(i)} {\tilde x}_a & (\forall i \in V)\\
\label{eq:gFlowReq}
            & \hspace{.6cm} \sum_{a \in \delta^-(i)} g_a - \sum_{a \in \delta^+(i)} g_a = 1 & (\forall i \in V_R \setminus \{1\}) \\
\label{eq:gFlowNonReq}
            & \hspace{.6cm} \sum_{a \in \delta^-(i)} g_a - \sum_{a \in \delta^+(i)} g_a = 0 & (\forall i \in V \setminus V_R)\\
\label{eq:gFlowBounds}
             & \hspace{-1cm} 0 \le g_a \le (n_R - 1){\tilde x}_a & (\forall a \in A)\\
\label{eq:sparseSCFbin}
            & {\tilde x}_a \in \{0,1\} & (\forall a \in A).
\end{eqnarray}
The constraints (\ref{eq:Xbounds}) ensure that the salesman departs from each required
node at least once, and the constraints (\ref{eq:xFlow}) ensure that the salesman departs
from each node as many times as he arrives.  The constraints (\ref{eq:gFlowReq}) impose
that one unit of the commodity is delivered to each required node, and the constraints
(\ref{eq:gFlowNonReq}) ensure that the amount of commodity on board when leaving a
non-required node is equal to the amount when arriving. The bounds (\ref{eq:gFlowBounds})
ensure that, if any of the commodity passes along an arc, then that arc appears in
the tour.

This formulation contains ${\cal O}(|E|)$ variables and ${\cal O}(|E|)$ constraints.

Using a technique due to Gouveia \cite{G95}, we can project this formulation into
the space of the $x$ variables:
\begin{theorem} \label{th:projection}
Let $({\tilde x}^*,g^*) \in [0,1]^{|A|} \times \RR_+^{|A|}$ be a feasible solution to the LP
relaxation of the formulation (\ref{eq:sparseSCFobj})--(\ref{eq:sparseSCFbin}). Let $x^*$ be
the corresponding point in $[0,2]^{|E|}$ defined by setting
$x^*_{ij} = {\tilde x}^*_{ij} + {\tilde x}^*_{ji}$ for all $\{i,j\} \in E$. Then $x^*$
satisfies all of the following linear inequalities:
\begin{equation} \label{eq:weak-con}
\sum_{e \in \delta(S)} x_e \ge 2 \frac{|S \cap V_R|}{n_R-1}
\qquad (\forall S \subset V \setminus \{1\}: S \cap V_R \ne \emptyset).
\end{equation}
\end{theorem}
\begin{proof}
If we sum the constraints (\ref{eq:gFlowReq}) over all $i \in S \cap V_R$, together with the
constraints (\ref{eq:gFlowNonReq}) over all $i \in S \setminus V_R$, we obtain:
\[
\sum_{a \in \delta^-(S)} g_a = \sum_{a \in \delta^+(S)} g_a + |S \cap V_R|.
\]
Together with the bounds (\ref{eq:gFlowBounds}), this implies:
\begin{equation} \label{eq:step1}
(n_R - 1) \sum_{a \in \delta^-(S)} {\tilde x}_a \ge |S \cap V_R|,
\end{equation}
Now, the equations (\ref{eq:xFlow}) imply:
\begin{equation} \label{eq:step2}
\sum_{a \in \delta^-(S)} {\tilde x}_a = \sum_{a \in \delta^+(S)} {\tilde x}_a.
\end{equation}
>From (\ref{eq:step1}) and (\ref{eq:step2}) we obtain:
\[
\sum_{a \in \delta^-(S) \cup \delta^+(S)} {\tilde x}_a
\ge 2 \frac{|S \cap V_R|}{n_R-1}.
\]
The result then follows from the construction of $x^*$.
\end{proof}

Note that the inequalities (\ref{eq:weak-con}) are weaker than the connectivity
inequalities (\ref{eq:con}). As a result, the lower bound associated with the
SCF formulation (\ref{eq:sparseSCFobj})--(\ref{eq:sparseSCFbin}) cannot be
better than the one associated with Fleischmann's formulation
(\ref{eq:stsp-obj})--(\ref{eq:genint}).

\subsection{Strengthened single-commodity flow formulation} \label{sub:newSCF2}

It is possible to strengthen the SCF formulation given in the previous subsection.
Note that one can assume that, if any required node is visited more than once
by the salesman, then the commodity is delivered on the first visit. Accordingly,
for each node $i \in V \setminus \{1\}$, let $r_i$ be the minimum number of
required nodes (not including the depot) that the salesman must have visited
when he leaves $i$ for the first time. Also, by convention, let $r_1 = 0$. (Note
that one can compute $r_i$ for all $i \in V \setminus \{1\}$ efficiently, using
Dijkstra's single-source shortest-path algorithm \cite{D59}). Now, the constraints
(\ref{eq:gFlowBounds}) can be replaced with the following stronger constraints:
\begin{equation} \label{eq:gFlowBounds2}
0 \le g_{ij} \le (n_R - r_i - 1) {\tilde x}_{ij} \qquad (\forall (i,j) \in A).
\end{equation}
This makes the projection into $x$-space stronger, as expressed in the
following theorem:
\begin{theorem} \label{th:projection2}
Let $({\tilde x}^*,g^*)$ be a feasible solution to the LP relaxation of the
formulation (\ref{eq:sparseSCFobj})--(\ref{eq:gFlowNonReq}),
(\ref{eq:sparseSCFbin}), (\ref{eq:gFlowBounds2}). Also, for any set
$S \subseteq V \setminus \{1\}$ such that $S \cap V_R \ne \emptyset$, let
$T(S)$ be the set of all nodes that are not in $S$ but are adjacent to at
least one node in $S$. Finally, define $L(S) = \min_{i \in T} r_i$ and
$U(S) = \max_{i \in T} r_i$. Then, the point $x^*$ corresponding to
$({\tilde x}^*,g^*)$ satisfies the following inequality for all such sets
$S$ and for $k = L(S), \ldots, U(S)$:
\begin{equation} \label{eq:weak-con2}
(n_R - k - 1) \sum_{e \in \delta(S)} x_e 
+ 2 \sum_{\{i,j\} \in \delta(S): j \in S}
\max \left\{ 0, k - r_i \right\} x_{ij}
\ge 2 |S \cap V_R|.
\end{equation}
\end{theorem}
\begin{proof}
As in the proof of Theorem \ref{th:projection}, the constraints
(\ref{eq:gFlowReq}) and (\ref{eq:gFlowNonReq}) imply:
\[
\sum_{a \in \delta^-(S)} g_a = \sum_{a \in \delta^+(S)} g_a + |S \cap V_R|.
\]
Using the strengthened bounds (\ref{eq:gFlowBounds2}), this implies:
\[
\sum_{(i,j) \in \delta^-(S)} (n_R - r_i - 1) {\tilde x}_{ij} \ge |S \cap V_R|.
\]
We can re-write this as:
\[
(n_R - k - 1) \sum_{(i,j) \in \delta^-(S)} {\tilde x}_{ij}
+ \sum_{(i,j) \in \delta^-(S)} (k - r_i) {\tilde x}_{ij}
\ge |S \cap V_R|.
\]
Together with non-negativity on $\tilde x$ this implies:
\[
(n_R - k - 1) \sum_{(i,j) \in \delta^-(S)} {\tilde x}_{ij}
+ \sum_{(i,j) \in \delta^-(S)} \max \{ 0, k - r_i \}
({\tilde x}_{ij} + {\tilde x}_{ji}) \ge |S \cap V_R|.
\]
The result then follows from the identity (\ref{eq:step2}) and the construction
of $x^*$.
\end{proof}

Our experiments on small instances lead us to conjecture that the inequalities
(\ref{eq:weak-con2}), together with the bounds $x \in [0,2]^{|E|}$, give a
complete description of the projection into $x$-space.

Note that, if one sets $k = L(S)$ in Theorem \ref{th:projection2}, one
obtains the following family of inequalities:
\[
\sum_{e \in \delta(S)} x_e \ge 2 \frac{|S \cap V_R|}{(n_R - L(S) - 1)}
\qquad (\forall S \subset V \setminus \{1\}: S \cap V_R \ne \emptyset).
\]
Since $L(S)$ cannot exceed $n_R - 1 - |S \cap V_R|$, these inequalities are
intermediate in strength between the inequalities (\ref{eq:weak-con}) and
the connectivity inequalities (\ref{eq:con}). Accordingly, we conjecture that
the lower bound from the strengthened SCF formulation always lies between the
one from the original SCF formulation and the one from Fleischmann's
formulation.

We also remark that one could tighten the constraints (\ref{eq:gFlowBounds2})
further for the arcs that are incident on the depot. Indeed, in an optimal
solution, the salesman would never depart from the depot without at least one
unit of the commodity, and would never arrive at the depot with more than
$n_R-2$ units of the commodity. One can check, however, that this further
tightening in the $({\tilde x},g)$-space does not lead to any improvement in
the resulting valid inequalities in the $x$-space.

\subsection{Multi-commodity flow formulation} \label{sub:newMCF}

Similar to the MCF formulation for the standard TSP, we assume that the salesman
leaves the depot (node $1$) with one unit of commodity for each required node.
Accordingly, let the binary variable $f_a^k$ be $1$ if and only if commodity $k$
passes through arc $a$, for every $k \in V_R \setminus \{1\}$ and $a \in A$. The
resulting formulation then consists of minimising (\ref{eq:sparseSCFobj}) subject
to the following constraints:
\begin{eqnarray}
\label{eq:MCFreq}
  & \hspace{-3.7cm} \sum_{a \in \delta^+(i)}{\tilde x}_a \ge 1 & (\forall i \in V_R)\\
\label{eq:MCFin.ou}
  & \hspace{-2cm} \sum_{a \in \delta^+(i)}{\tilde x}_a
  = \sum_{a \in \delta^-(i)}{\tilde x}_a & (\forall i\in V)\\
\label{eq:fFlowNonReq}
  & \hspace{-1.2cm} \sum_{a \in \delta^-(i)} f_a^k - \sum_{a \in \delta^+(i)} f_a^k
  = 0 & (\forall i \in V \setminus \{1\}; k \in V_R \setminus \{1,i\})\\
\label{eq:fFlowIn}
  & \hspace{-1.2cm} \sum_{a \in \delta^-(k)} f_a^k - \sum_{a \in \delta^+(k)} f_a^k
  = 1 & (\forall k \in V_R \setminus \{1\})\\
\label{eq:fFlowOut}
& \hspace{-.8cm} \sum_{a \in \delta^-(1)} f_a^k - \sum_{a \in \delta^+(1)} f_a^k
  = -1 & (\forall k \in V_R \setminus \{1\})\\
\label{eq:MCFbound}
  & \hspace{-5cm} {\tilde x}_a \ge f_a^k & (\forall a \in A ; k \in V_R \setminus \{1\})\\
& \hspace{-4.5cm} {\tilde x}_a \in \{0,1\} & (\forall a \in A)\\
& \hspace{-4.5cm} f_a^k \in \{0,1\} & (\forall a \in A \& k \in V_R \setminus \{1\}).
\end{eqnarray}
The constraints are interpreted along similar lines to those of the formulations
already seen.

This MCF formulation has ${\cal O}(n_R|E|)$ variables and ${\cal O}(n_R|E|)$ constraints.
As for the projection into the space of $x$ variables, we have the following result:
\begin{proposition}
Let $({\tilde x}^*,f^*)$ be a feasible solution to the LP relaxation of the MCF
formulation. Let $x^*$ be the corresponding point in $[0,2]^{|E|}$ defined by setting
$x^*_{ij} = {\tilde x}^*_{ij} + {\tilde x}^*_{ji}$ for all $\{i,j\} \in E$. Then $x^*$
satisfies all of the the connectivity inequalities (\ref{eq:con}). 
\end{proposition}
\begin{proof}
For a fixed node $k \in V_R \setminus \{1\}$, the constraints
(\ref{eq:fFlowNonReq})--(\ref{eq:MCFbound}), together with the well-known
{\em max-flow min-cut}\/ theorem \cite{FF56} imply the following exponentially-large
family of inequalities:
\[
\sum_{a \in \delta^+(S)} {\tilde x}^*_a \ge 1 \qquad
(\forall S \subset V_R \setminus \{1\}: k \in S).
\]
The equations (\ref{eq:MCFin.ou}) then imply:
\[
\sum_{a \in \delta^+(S) \cup \delta^-(S)} {\tilde x}^*_a \ge 2 \qquad
(\forall S \subset V_R \setminus \{1\}: k \in S).
\]
Next, the relationship between ${\tilde x}^*$ and $x^*$ gives 
\[
\sum_{e \in \delta(S)} x^*_e \ge 2 \qquad
(\forall S \subset V_R \setminus \{1\}: k \in S).
\]
Applying this for all $k \in V_R \setminus \{1\}$ yields the result.
\end{proof}

This result implies that the lower bound from this MCF formulation is no
worse than the one from Fleischmann's formulation. We conjecture that
the two bounds are equal.

\section{Time-Staged Formulations of the STSP} \label{se:time}

In this section, we adapt the TS formulation for the standard TSP, mentioned in
Subsection \ref{sub:lit2}, to the Steiner case. A simple formulation is
presented in the following subsection. A method to reduce the number of
variables is presented in Subsection \ref{sub:newTS2}. Then, in Subsection
\ref{sub:summary}, we evaluate the total number of variables and constraints
in each of the formulations that we have considered.

\subsection{An initial time-staged formulation} \label{sub:newTS1}

In this context, it is natural to have one time stage for each time that an
edge of $G$ is traversed (in either direction).  In terms of the classical
STSP formulation given in Subsection \ref{sub:lit3}, the total number of time
stages will then be equal to $\sum_{e \in E} x_e$.  The problem here is that
we do not know this value in advance. Observe, however, that
Lemma~\ref{le:once} implies that it cannot exceed $2|E|$.

Now, let $A$ be defined as in Subsection \ref{sub:notation}, and recall that
$|A| = 2|E|$. For all $a \in A$ and all $1 \le k \le |A|$, let the binary
variable $r_a^k$ take the value $1$ if and only if arc $a$ is the $k$th arc to
be traversed in the tour. Our TS formulation for the STSP is as follows:
\begin{eqnarray}
\min              & \hspace{-6mm} \sum_{k = 1}^{|A|} \sum_{a \in A} c_a r_a^k & \\
\label{eq:newTS1}
\mbox{s.t.}       & \hspace{-6mm} \sum_{a \in \delta^+(1)} r_a^1 = 1 & \\ 
\label{eq:newTS2} & \hspace{-6mm} r_a^1 = 0 & \hspace{-6mm} (a \in A \setminus \delta^+(1)) \\
\label{eq:newTS3} & \hspace{-6mm} \sum_{k = 1}^{|A|} \sum_{a \in \delta^+(1)} r_a^k
                    = \sum_{k = 1}^{|A|} \sum_{a \in \delta^-(1)} r_a^k & \\
\label{eq:newTS4} & \hspace{-6mm} \sum_{k = 1}^{|A|} \sum_{a \in \delta^+(i)} r_a^k \ge 1
                  & \hspace{-6mm} (\forall i \in V_R)\\
\label{eq:newTS5} & \hspace{-6mm} \sum_{a \in \delta^-(i)} r_a^k = \sum_{a \in \delta^+(i)} r_a^{k+1} 
                  & \hspace{-6mm} (\forall i\in V ; k = 1, \ldots, |A|-1)\\
\label{eq:newTS6} & \hspace{-6mm} r_a^k \in \{0,1\}
                  & \hspace{-6mm} (\forall a \in A, k = 1, \ldots, |A|).
\end{eqnarray}
Constraints (\ref{eq:newTS1}) and (\ref{eq:newTS2}) ensure that the salesman departs
from the depot in the first time stage, and constraint (\ref{eq:newTS3}) ensures that
he arrives at the depot as many times as he leaves it. Constraints (\ref{eq:newTS4})
ensure that each required node is visited at least once. Constraints (\ref{eq:newTS5})
ensure that, if the salesman arrives at a non-depot node in any given time stage,
then he must depart from it in the subsequent time stage. Finally, constraints
(\ref{eq:newTS6}) are the usual binary conditions.

This TS formulation has ${\cal O}(|E|^2)$ variables and ${\cal O}(n|E|)$ constraints.
We conjecture that the lower bound from this TS formulation always lies between the
one from our strengthened SCF formulation and the one from Fleischmann's formulation.

\subsection{Bounding the number of edge traversals} \label{sub:newTS2}

Clearly, one could reduce the number of variables and constraints in the above TS
formulation if one had a better upper bound on the total number of times that the
salesman traverses an edge of $G$. The following theorem provides such a bound:

\begin{theorem}\label{thm:STSP-tstages-bd}
  For every instance of the STSP which has a solution, there
  exists an optimal solution in which the total number of edge
  traversals (in either direction) does not exceed $2(|V|-1)$.
\end{theorem}

For the proof of this theorem, we will use the following lemma.

\begin{lemma}\label{lem:nondisconnectingcycle}
  If~$H$ is a connected graph on~$k$ nodes which has more than
  $2(k-1)$ edges, then there exists a cycle~$C$ in~$H$ such that the
  graph arising when the edges of~$C$ are deleted from~$H$ is still
  connected.
\end{lemma}
\begin{proof}
  Let~$T$ be a spanning tree in~$H$, and let $H'$ be the graph
  resulting if the edges of~$T$ are deleted from~$H$.  For the
  number~$\ell'$ of edges of~$H'$ we have $\ell' = \ell - (k-1)$,
  which, by the hypothesis in the lemma, is greater than $k -1$.
  Clearly, the number of nodes of~$H'$ is equal to~$k$.

  Now, let~$T'$ be a spanning forest in~$H'$.  Note that~$H'$ may fail
  to be connected.  Firstly, if one of the connected components
  of~$H'$ contains an edge~$e$ other than those in~$T'$, then let $C$
  be the cycle defined by taking~$e$ and the path in~$T'$ connecting
  the end-nodes of~$e$.  Clearly, deleting the edges of~$C$ from
  $H$ leaves a connected graph because connectivity is assured by the
  tree~$T$.

  But, secondly, it is impossible that all connected components
  of~$H'$ contain no other edges except those in~$T'$: In that case,
  $H'$ would be a forest, and hence have at most $k-1$ edges.  But
  the number of edges of~$H'$ is greater than $k-1$, a contradiction.
\end{proof}

We can now complete the proof of the theorem.

\begin{proofof}{Proof of Theorem~\ref{thm:STSP-tstages-bd}}
  Let~$x$ be an optimal solution to the STSP, which has, among
  all optimal solutions, the smallest number of edge traversals.

  Construct a graph~$H$ by starting with the node set~$V$, and
  precisely~$x_e$ copies of the edge~$e$, for all $e\in E$.  Then
  delete every isolated node from~$H$.  The number of nodes~$k$
  of~$H$ is at most~$|V|$, and the number of edges is $\ell :=
  \sum_{e\in E} x_e$.

  For the sake of contradiction, we assume that $\ell > 2(|V|-1)$.  If
  that is the case, then Lemma~\ref{lem:nondisconnectingcycle}, is
  applicable.  Let~$C$ be a cycle with the property given in the
  lemma, and let~$F$ be its edge set.  For every $e\in E$, denote
  by~$y_e$ the number of times the edge~$e$ occurs in~$C$.  The fact
  that after deleting the edges of~$C$ from~$H$, a connected graph
  remains, implies that $x - y$ is a solution to the STSP, whose total
  cost is at most that of~$x$.  Thus, $x-y$ is an optimal solution in
  which the total number of edge traversals is smaller than in~$x$,
  contradicting the choice of~$x$.

  Thus, we conclude that $\sum_e x_e = \ell \le 2(|V|-1)$.
\end{proofof}

An immediate consequence of this theorem is that one does not need to define
the variables $r_a^k$ in the TS formulation when $k > 2(|V|-1)$. The
constraints in which $k > 2(|V|-1)$ can be dropped as well. As a result, the
number of variables and constraints in the TS formulation can be reduced to
${\cal O}(n|E|)$ and ${\cal O}(n^2)$, respectively. We conjecture that this
reduction in size has no effect on the associated lower bound.

\subsection{Summary} \label{sub:summary}

Table \ref{tab:size} displays, for each of the STSP formulations that we have
considered, bounds on the total number of variables and constraints. Here,
`classical' refers to the formulation of Fleischmann \cite{F85} mentioned in
Subsection \ref{sub:lit3}, `SCF' refers to either of the single-commodity flow
formulations given in Subsections \ref{sub:newSCF1} and \ref{sub:newSCF2},
`MCF' refers to the multi-commodity flow formulation given in Subsection
\ref{sub:newMCF}, `TS1' refers to the time-staged formulation given in
Subsection \ref{sub:newTS1}, and `TS2' refers to the reduced time-staged
formulation given in Subsection \ref{sub:newTS2}.

\begin{table}
\centering  
\begin{tabular}{cccccc}  
\hline
Formulation  & Classical          & SCF           & MCF                & TS1               & TS2 \\
\hline
Variables    & $|E|$              & ${\cal O}(|E|)$ & ${\cal O}(n_R|E|)$ & ${\cal O}(|E|^2)$ & ${\cal O}(n|E|)$ \\
Constraints & ${\cal O}(2^{n_R})$ & ${\cal O}(|E|)$ & ${\cal O}(n_R|E|)$ & ${\cal O}(n|E|)$  & ${\cal O}(n^2)$ \\
\hline
\end{tabular}
\caption{Alternative STSP formulations and their size}
\label{tab:size}
\end{table}

Observe that, in the case of real road networks, the graph $G$ is typically very
sparse, and we have $|E| = {\cal O}(|V|)$. Then, any of the new formulations could
potentially be used in practice. We would recommend using MCF or TS2 for small or
medium-sized instances, due to the relative tightness of the bound, and SCF for
large instances, due to the extremely small number of variables and constraints.

Observe that we have not adapted the MTZ formulation to the Steiner case. This is
because the MTZ formulation is based on the idea of determining the order in which
the nodes are visited. Since nodes can be visited multiple times in the Steiner case,
a unique order cannot be determined. As a result, it does not appear possible to
adapt the MTZ formulation. This is not a problem, though, given the extreme weakness
of the MTZ formulation mentioned in Subsection \ref{sub:lit2}.

\section{Some Related Problems} \label{se:other}

Many variants and extensions of the TSP have appeared in the literature, such as
the {\em Orienteering Problem}\/ (e.g., \cite{FDG05,FST02,GLV87}), the
{\em Prize-Collecting}\/ TSP (e.g., \cite{B89,B02,FDG05}), the {\em Capacitated
Profitable Tour Problem}\/ (e.g., \cite{FDG05,JPSP12}), the {\em Generalized}\/
TSP (e.g., \cite{FST02,SKGS69}), the TSP with {\em Time Windows}\/ (e.g.,
\cite{AFG01,DDGS95}) and the {\em Sequential Ordering Problem} \cite{E88}.
For each of these problems, it is easy to define a `Steiner' version. It suffices
to define the problem on a general graph $G=(V,E)$, designate node $1$ as the
`depot', define a set $V_R \subset V \setminus \{1\}$ of `customer' nodes, permit
edges to traversed more than once if desired, and permit nodes to be visited more
than once if desired.

In this section, we explore possible ways to formulate these other problems of
`Steiner' type. For the sake of brevity, however, we restrict attention to three
specific problems, which we call the {\em Steiner Orienteering Problem}, the
{\em Steiner Capacitated Profitable Tour Problem}, and the {\em Steiner TSP with
Time Windows}. These are considered in the following three subsections.

\subsection{The Steiner Orienteering Problem}

We define the {\em Steiner Orienteering Problem}\/ (SOP) as follows. For
each $e \in E$, we are given a non-negative cost $c_e$. For each $i \in V_R$, we
are given a positive {\em revenue}\/ (or `prize') $p_i$. The nodes in $V_R$ do
not all have to be visited, but the revenue can only be collected from such a
node if that node is visited at least once. We are also given an upper bound $U$
on the total route cost. The task is to maximise the sum of the prizes collected,
subject to the upper bound.

Observe that Lemma \ref{le:once} applies to the SOP. To see this, let
$V^* \subset V_R$ be the set of nodes whose prizes are collected in the optimal
solution. The optimal solution is then also optimal for a STSP instance
defined on the same graph, but with $V_R$ set to $V^*$.

Knowing that Lemma \ref{le:once} applies, it is easy to adapt the classical
(non-compact) formulation of the STSP, presented in Subsection
\ref{sub:lit3}, to the SOP. For each $i \in V_R$, we define a new
binary variable $y_i$, taking the value $1$ if and only if the salesman
collects a prize from node $i$. We then change the objective function from
(\ref{eq:stsp-obj}) to:
\begin{equation} \label{eq:op-obj}
\max \sum_{i \in V_R} p_i y_i,
\end{equation}
replace the connectivity constraints (\ref{eq:con}) with:
\begin{equation} \label{eq:xycon}
\sum_{e \in \delta(S)} x_e \ge 2y_i \qquad
(i \in V_R, S \subseteq V \setminus \{1\}: i \in S),
\end{equation}
and add the route-cost constraint
\begin{equation} \label{eq:op-cost}
\sum_{e \in E} c_e x_e \le U.
\end{equation}

It is also easy to adapt the TS formulation of the STSP (Subsection
\ref{sub:newTS1}) to the SOP. It suffices to add the $y_i$ variables
mentioned above, change the objective function from (\ref{eq:newTS1}) to
(\ref{eq:op-obj}), add the route-cost constraint
\[
\sum_{k = 1}^{|A|} \sum_{a \in A} c_a r_a^k \le U,
\]
and replace the constraints (\ref{eq:newTS4}) with the constraints
\begin{equation} \label{eq:ts-xycon}
\sum_{k = 1}^{|A|} \sum_{a \in \delta^+(i)} r_a^k
\ge y_i \qquad (\forall i \in V_R).
\end{equation}
Moreover, Theorem \ref{thm:STSP-tstages-bd}, given in Subsection \ref{sub:newTS1},
applies to the SOP as well (for the same reason that Lemma \ref{le:once}
applies). So one can reduce the number of stages to $2(|V|-1)$, without losing any
optimal solutions.

It is also easy to adapt the MCF formulation of the STSP (Subsection
\ref{sub:newMCF}) to the SOP. It suffices to add the same $y_i$ variables,
change the objection function from (\ref{eq:sparseSCFobj}) to (\ref{eq:op-obj}),
change the right-hand sides of constraints (\ref{eq:MCFreq}) and (\ref{eq:fFlowIn})
from $1$ to $y_i$, change the right-hand sides of constraints (\ref{eq:fFlowOut})
from $-1$ to $-y_i$, and add the route-cost constraint:
\begin{equation} \label{eq:op-cost2}
\sum_{a \in A} c_a {\tilde x}_a \le U.
\end{equation}

As for the SCF formulation of the STSP (Subsection \ref{sub:newSCF1}),
there is an elegant way to adapt it to the SOP, which leads to an LP
relaxation with desirable properties. The key is to redefine the continuous
variables $g_a$, so that:
\begin{itemize}
\item if arc $a$ is traversed (i.e., ${\tilde x}_a = 1$), then $g_a$ represents
the total cost accumulated so far when the salesman begins to traverse the arc
\item if arc $a$ is not traversed (i.e., ${\tilde x}_a = 0$), then $g_a = 0$.
\end{itemize}
One this is done, one can introduce the same additional $y_i$ variables, and use
the objective function (\ref{eq:op-obj}), along with the following constraints:
\begin{eqnarray}
\label{eq:op1}
  & \sum_{a \in \delta^+(1)} \tilde{x}_a \ge 1 & \\
\label{eq:op2}
  & \sum_{a \in \delta^+(i)} \tilde{x}_a \ge y_i & (\forall i \in V_R) \\
\label{eq:op3}
  & \sum_{a \in \delta^+(i)} \tilde{x}_a = \sum_{a \in \delta^-(i)} \tilde{x}_a
  & (\forall i \in V)\\
\label{eq:op4}
  & \sum_{a \in \delta^+(i)} g_a - \sum_{a \in \delta^-(i)} g_a =
  \sum_{a \in \delta^-(i)} c_a \tilde{x}_a & (\forall i \in V \setminus \{1\}))\\
\label{eq:op5}
  & 0 \le g_a \le (U - c_a) \tilde{x}_a & (\forall a \in A)\\
\label{eq:op6}
  & \tilde{x}_a \in \{0,1\}             & (\forall a \in A)\\
\label{eq:op7}
  & y_i \in \{0,1\}                     & (\forall i \in V_R). 
\end{eqnarray}

We then have the following analogue of Theorem \ref{th:projection}:
\begin{proposition}
Let $({\tilde x}^*,g^*,y^*) \in [0,1]^{|A|} \times \RR_+^{|A|} \times [0,1]^{n_R}$
satisfy the constraints (\ref{eq:op1})--(\ref{eq:op5}). Let $x^* \in [0,2]^{|E|}$
be defined by setting $x^*_{ij} = {\tilde x}^*_{ij} + {\tilde x}^*_{ji}$ for
all $\{i,j\} \in E$. Then $(x^*,y^*)$ satisfies all of the following linear
inequalities:
\[
U \sum_{e \in \delta(S)} x_e \ge
\sum_{\{i,j\} \in E: \{i,j\} \cap S \ne \emptyset} c_e x_e \qquad
(\forall S \subset V \setminus \{1\}: S \cap V_R \ne \emptyset).
\]
\end{proposition}
\begin{proof}
Similar to the proof of Theorem \ref{th:projection}.
\end{proof}

As in the case of the STSP (Subsection \ref{sub:newSCF2}), it is
possible to strengthen this SCF formulation of the SOP. Indeed, if a
given arc $(i,j)$ is traversed, then the smallest value that $g_{ij}$ can take
is equal to the cost of the shortest path from the depot to node $i$. Similarly,
the largest value that $g_{ij}$ can take is equal to $U - c_{ij}$ minus the cost
of the shortest path from node $j$ to the depot. One can adjust the constraints
(\ref{eq:op5}) accordingly, and then derive a stronger projection result,
analogous to Theorem \ref{th:projection2}. We omit details, for the sake of
brevity.

\subsection{The Steiner Capacitated Profitable Tour Problem}\label{sub:SCPTP}

The {\em Steiner Capacitated Profitable Tour Problem}\/ (SCPTP) is
similar to the SOP, but with the following differences:
\begin{itemize}
\item We are given a positive {\em demand}\/ $q_i$ for each $i \in V_R$, in
addition to the revenue $p_i$.
\item If we wish to gain the revenue for a given $i \in V_R$, then we have to
deliver the demand $q_i$.
\item Instead of an upper bound $U$ on the route cost, we are given a vehicle
capacity $Q$, which does not exceed the sum of the demands. The total demand
of the serviced customers must not exceed $Q$.
\item The task is to find a tour of maximum total profit, where the profit
is defined as the sum of the revenues gained, minus the cost of the edges
traversed.
\end{itemize}

Observe that Lemma \ref{le:once} applies to the SCPTP, for the same
reason that it applies to the SOP. Then, one can easily adapt the
classical formulation of the STSP to the SCPTP. We use the same
binary variables $y_i$ as used in the previous subsection, change the objective
function from (\ref{eq:stsp-obj}) to
\[
\max \sum_{i\in V_R} p_i y_i - \sum_{e \in E} c_e x_e,
\]
replace the connectivity constraints (\ref{eq:con}) with the constraints
(\ref{eq:xycon}), and add the capacity constraint
\begin{equation} \label{eq:cptp-cap}
\sum_{i\in V_R} q_i y_i \le Q.
\end{equation}

One can adapt the TS formulation in a similar way. It suffices to add the same
$y_i$ variables, add the capacity constraint (\ref{eq:cptp-cap}), change the
objective function (\ref{eq:newTS1}) to
\[
\max \sum_{i \in V_R} p_i y_i - \sum_{k=1}^{|A|} \sum_{a \in A} c_a r_a^k, 
\]
and replace the constraints (\ref{eq:newTS4}) with the constraints
(\ref{eq:ts-xycon}). Moreover, Theorem \ref{thm:STSP-tstages-bd} is again
applicable, and one can reduce the number of stages to $2(|V|-1)$.

As for the SCF formulation, we propose again to redefine the continuous
variables $g_a$. Now, $g_a$ represents the total load (if any) that is carried
along the arc $a$. Then, again using the additional $y_i$ variables, it suffices
to:
\begin{equation} \label{eq:cptp-scf-obj}
\max \sum_{i\in V_R} p_i y_i - \sum_{a \in A} c_a \tilde{x}_a
\end{equation}
subject to the following constraints:
\begin{eqnarray}
\label{eq:Depleave}
  & \sum_{a \in \delta^+(1)} \tilde{x}_a \ge 1  & \\
  & \sum_{a \in \delta^+(i)} \tilde{x}_a \ge y_i & (\forall i \in V_R) \\
\label{eq:xFlow2}
  & \sum_{a \in \delta^+(i)} \tilde{x}_a = \sum_{a \in \delta^-(i)} \tilde{x}_a
  & (\forall i \in V)\\
\label{eq:weird-cap}
  & \sum_{a \in \delta^+(1)} g_a - \sum_{a \in \delta^-(1)} g_a \le Q & \\
\label{eq:gFlowReq2}
  & \sum_{a \in \delta^-(i)} g_a - \sum_{a \in \delta^+(i)} g_a = q_i y_i
  & (\forall i \in V_R)\\
\label{eq:gFlowNonReq2}
  & \sum_{a \in \delta^-(i)} g_a - \sum_{a \in \delta^+(i)} g_a = 0
  & (\forall i \in V \setminus (V_R \cup \{1\}))\\
& 0 \le g_a \le Q \tilde{x}_a & (\forall a \in A)\\
& \tilde{x}_a \in \{0,1\} & (\forall a \in A)\\
& y_i \in \{0,1\} & (\forall i \in V_R). 
\end{eqnarray}

The analogue of Theorem \ref{th:projection} is now as follows:
\begin{proposition}
Let $({\tilde x}^*,g^*,y^*) \in [0,1]^{|A|} \times \RR_+^{|A|} \times [0,1]^{n_R}$
satisfy (\ref{eq:Depleave})--(\ref{eq:gFlowNonReq2}). Let $x^* \in [0,2]^{|E|}$
be defined by setting $x^*_{ij} = {\tilde x}^*_{ij} + {\tilde x}^*_{ji}$ for
all $\{i,j\} \in E$. Then $(x^*,y^*)$ satisfies all of the following linear
inequalities:
\begin{equation} \label{eq:y-multistar}
\sum_{e \in \delta(S)} x_e \ge 2 \frac{\sum_{i \in S \cap V_R} q_i y_i}{Q}
\qquad (\forall S \subseteq V \setminus \{1\}: S \cap V_R \ne \emptyset).
\end{equation}
\end{proposition}
\begin{proof}
Similar to the proof of Theorem \ref{th:projection}.
\end{proof}
Moreover, if one sums together the constraints
(\ref{eq:weird-cap})--(\ref{eq:gFlowNonReq2}), one obtains the capacity
constraint (\ref{eq:cptp-cap}). So the capacity constraint does not need to be
added to this SCF formulation.

As for the MCF formulation described in Subsection \ref{sub:newMCF}, we propose to
redefine the binary variables $f_a^k$ to be $1$ if and only if $q_k$ units of
commodity $k$ pass through arc $a$. Then, using the same additional $y_i$ variables,
it suffices to change the objection function to (\ref{eq:cptp-scf-obj}),
change the right-hand sides of constraints (\ref{eq:MCFreq}) and (\ref{eq:fFlowIn})
from $1$ to $y_i$, change the right-hand sides of constraints (\ref{eq:fFlowOut})
from $-1$ to $-y_i$, and add the constraints:
\[
\sum_{k \in V_R} q_k f_a^k \le Q {\tilde x}_a \qquad (\forall a \in A).
\]
It can be shown that the projection of this formulation into $(x,y)$ space
satisfies the inequalities (\ref{eq:xycon}) and (\ref{eq:y-multistar}), along with
the capacity constraint (\ref{eq:cptp-cap}). We omit the details for brevity.

\subsection{The Steiner TSP with Time Windows}

Finally, we define the {\em Steiner Traveling Salesman Problem with Time
Windows}\/ (STSPTW) as follows. As before, we are given a non-negative
cost $c_e$ each $e \in E$. For each $e \in E$, we are given a non-negative
{\em traversal time}\/ $t_e$. Moreover, for each $i \in V_R$, we are given a
non-negative {\em servicing time}\/ $s_i$, along with a {\em time window}\/
$[a_i, b_i]$. Finally, we are given a positive time $T$ by which the vehicle
must return to the depot. All nodes in $V_R$ must be visited at least once.
On one such visit, the customer must receive service. The time at which service
begins must lie between $a_i$ and $b_i$. The task is to minimise the cost of
the tour. We assume without loss of generality that the vehicle departs from
the depot at time zero. We also assume that the vehicle is permitted to wait
at any customer node, if it arrives at that node before service is due to begin.

Perhaps surprisingly, the situation here is completely different from those of
the previous two subsections. To be specific:
\begin{itemize}
\item Lemma \ref{le:once} does not apply. To see this, set $V = \{1, \ldots, 4\}$,
$V_R = \{2,3,4\}$ and $E = \{ \{1,2\}, \{1,3\}, \{2,4\} \}$, set
$c_e = t_e = 1$ for all $e \in E$, set $s_i = 1$ for $i \in \{2,3,4\}$, and set
$a_2 = b_2 = 1$, $a_3 = b_3 = 3$, and $a_4 = b_4 = 6$. The unique optimal solution
is for the salesman to service nodes $2, 3$ and $4$ in that order, and then return
to the depot. In this solution, the edge $\{1,2\}$ is traversed $4$ times.
\item Theorem \ref{thm:STSP-tstages-bd} does not apply either. In the same example,
the total number of edge traversals is $8$, whereas $2(|V|-1)$ is only 6.
\item In fact it is not even true that the total number of edge traversals is
bounded by $2|E|$, as the same example shows.
\item The only thing that one can say in general seems to be that the total
number of edge traversals is bounded by $(n_R+1)(|V|-1)$. (This is so since
the maximum number of edge traversals between two successive occasions of service,
or between a service and the vehicle leaving or returning to the depot, will
never exceed $|V|-1$ in an optimal solution.)
\end{itemize}
For these reasons, it does not seem possible to adapt the classical, SCF or MCF
formulations to the STSPTW, and it does not seem desirable to adapt the TS
formulation, since one would need $(n_R+1)(|V|-1)$ time stages.

On a more positive note, however, there exists a compact formulation of the STSPTW
that uses only ${\cal O}(n_R|E|)$ variables and constraints. The necessary
variable definitions are as follows. For every $a \in A$ and $k = 0, \ldots, n_R$,
let the binary variable $\tilde{x}_a^k$ take the value $1$ if and only if the
salesman traverses arc $a$ after having serviced exactly $k$ customers so far.
Also let $g_a^k$ be a non-negative continuous variable representing the total time
that has elapsed when the salesman starts to traverse arc $a$, having exactly
serviced $k$ customers, or $0$ if no such traversal occurs. Finally, for all
$i \in V_R$ and $k =1, \ldots, n_R$, let the binary variable $y_i^k$ take the
value $1$ if and only if customer $i$ is the $k$th customer to be serviced.

The objective function is simply:
\[
\min \sum_{k = 0}^{n_R} \sum_{a \in A} c_a \tilde{x}_a^k.
\]
To ensure that each required node is serviced exactly once, we have the following
constraints:
\begin{eqnarray*}
 & \sum_{k=1}^{n_R} y_i^k = 1 &  (\forall i \in V_R)\\ 
 & \sum_{i \in V_R} y_i^k  = 1 & (k=1, \ldots, n_R).
\end{eqnarray*}
To ensure that the vehicle departs from and returns to the depot a correct
number of times, we have:
\begin{eqnarray*}
 & \sum_{a \in \delta^+(1)}\tilde{x}_a^0 = 1  &\\
 & \sum_{a \in \delta^-(1)} \tilde{x}_a^k = \sum_{a \in \delta^+(1)} \tilde{x}_a^k 
   & (\forall k = 1, \ldots, n_R-1)\\
 & \sum_{a \in \delta^-(1)}\tilde{x}_a^{n_R} = 1.  &
\end{eqnarray*}
Then, to ensure that the vehicle departs from each non-depot node as many times as
it arrives, we have:
\begin{eqnarray*}
 & \sum_{a \in \delta^-(i)} \tilde{x}_a^0 =
   y_i^1 + \sum_{a \in \delta^+(i)}\tilde{x}_a^0  &(\forall i \in V_R)\\
 & y_i^k + \sum_{a \in \delta^-(i)} \tilde{x}_a^k =
   y_i^{k+1} + \sum_{a \in \delta^+(i)}\tilde{x}_a^k 
     &(\forall i \in V_R, \,  k = 1, \ldots, n_R-1)\\
 & y_i^{n_R} + \sum_{a \in \delta^-(i)} \tilde{x}_a^{n_R} =
   \sum_{a \in \delta^+(i)}\tilde{x}_a^{n_R} &(\forall i \in V_R)\\
 & \sum_{a \in \delta^-(i)} \tilde{x}_a^k = \sum_{a \in \delta^+(i)} \tilde{x}_a^k 
   & (\forall i\in V \setminus V_R \cup \{1\}, \, k = 0, \ldots, n_R).
\end{eqnarray*}
Next, to ensure that the $g_a^k$ variables take the value that they should, we
add the following constraint for $i \in V_R$ and for $k = 0, \ldots, n_R-1$:
\[
\sum_{a \in \delta^+(i)} g_a^{k+1} \ge \sum_{a \in \delta^-(i)} g_a^k
+ \sum_{a \in \delta^-(i)} t_a \tilde{x}_a^k + s_i y_i^{k+1},
\]
and the following constraint for $i \in V \setminus V_R$ and for
$k = 0, \ldots, n_R$:
\[
\sum_{a \in \delta^+(i)} g_a^k \ge \sum_{a \in \delta^-(i)} g_a^k
+ \sum_{a \in \delta^-(i)} t_a \tilde{x}_a^k.
\]
Moreover, to ensure that the time windows are obeyed, we add the following
constraints:
\begin{eqnarray*}
 & \sum_{a \in \delta^+(i)} g_a^k \ge (a_i+s_i) y_i^k
 & (\forall i \in V_R, k = 1, \ldots, n_R) \\
 & \sum_{a \in \delta^-(i)} g_a^{k-1} + \sum_{a \in \delta^-(i)} t_a \tilde{x}_a^{k-1}
 \le T - (T - b_i) y_i^k 
 & (\forall i \in V_R, \, k = 1, \ldots, n_R).
\end{eqnarray*}
Finally, we have the trivial constraints:
\begin{eqnarray*}
& \tilde{x}_a^k \in \{0,1\} & (\forall a \in A, \, k = 0, \ldots, n_R)\\
 & y_i^k \in \{0,1\} & (\forall i \in V_R, \, k = 1, \ldots, n_R)\\
 & 0 \le g_a^k \le T \tilde{x}_a^k & (\forall a \in A, \, k = 0, \ldots, n_R).
\end{eqnarray*}

As stated above, this formulation has ${\cal O}(n_R|E|)$ variables and
constraints. We leave the existence of a significantly smaller compact
formulation as an open question.

We remark that it is {\em not}\/ easy to convert the STSPTW into the
standard TSPTW by computing all-pairs shortest paths. This is because a
cheapest path between two nodes is not always the same as the quickest path.
We will address this issue in detail in another paper \cite{LS12}.

\section{Concluding Remarks} \label{se:conclusion}

Our motive for looking at the `Steiner' version of the TSP and its variants
was that many real-life vehicle routing problems are defined on road networks,
rather than complete graphs as normally assumed in the literature. Moreover,
`compact' formulations are of interest, not only for their elegance, but
also because one can just feed them into a standard branch-and-bound solver,
without having to implement complex solution methods such as branch-and-cut.

We have seen that the classical, single-commodity flow, multi-commodity flow
and time-staged formulations of the Traveling Salesman Problem can all be
adapted to the Steiner Traveling Salesman Problem, the Steiner Orienteering
Problem and the Steiner Capacitated Profitable Tour Problem. In some cases,
we can characterise the projections of the resulting LP relaxations into the
space of the `natural' variables. Moreover, in some cases, the formulations
can be easily strengthened, without increasing their size.

On the other hand, it does not seem possible to adapt the above formulations
to the Steiner Traveling Salesman Problem with Time Windows. Nevertheless,
we have produced a compact formulation of this problem which is of reasonable
size.

We believe that all of the formulations presented in this paper are potentially
of practical use. Possible topics for future research would be the derivation
of smaller and/or stronger compact formulations for the problems mentioned, the
derivation of useful compact formulations for the Steiner version of other
variants of the TSP, and exploring the potential of extending the approach
to problems with multiple vehicle and/or depots.

\end{document}